\documentclass[a4paper,12pt]{article}
\usepackage[T2A]{fontenc} 

\usepackage[english]{babel}
\usepackage{mathrsfs}
\usepackage[top=2cm,left=3cm,right=1.5cm,bottom=2cm]{geometry}
\usepackage{amsmath}
\usepackage{amsfonts}
\usepackage{amssymb}
\usepackage{amsthm} 
\usepackage{enumerate}

\numberwithin{equation}{subsection}

\theoremstyle{plain}
\newtheorem{theorem}{Theorem}

\newtheorem{corollary}[theorem]{Corollary}

\newtheorem{assertion}[theorem]{Assertion}

\theoremstyle{definition}
\newtheorem{definition}{Definition}
\newtheorem{example}{Example}

\theoremstyle{remark}
\newtheorem{remark}{Remark}

\begin{document}



\begin{center}\textbf{NORMAL EXTENSIONS}\end{center}

\begin{center}\textbf{B.\,N. Biyarov}\end{center} 
\\ \\
\textbf{Abstract.}
Let $L_0$ be a densely defined minimal linear operator in a Hilbert space $H$. We prove theorem that if there exists at least one correct extension $L_S$ of $L_0$ with the property $D(L_S)=D(L_S^*)$, then we can describe all correct extensions $L$ with the property $D(L)=D(L^*)$.
We also prove that if $L_0$ is formally normal and there exists at least one correct normal extension $L_N$, then we can describe all correct normal extensions $L$ of $L_0$.  As an example, the Cauchy-Riemann operator is given.


\subsection{Introduction}
\label{subsec1}

Let us present some definitions, notation, and terminology.

In a Hilbert space $H$, we consider a linear operator $L$ with  domain $D(L)$ and range $R(L)$. 
By the \textit{kernel} of the operator $L$ we mean the set
\[\mbox{Ker}\,L=\bigl\{f\in D(L): \; Lf=0\bigl\}.\]

\begin{definition}
\label{def1}  
An operator $L$ is called a \textit{restriction} of an operator $L_1$, and $L_1$ is called an \textit{extension} of an operator $L$, briefly $L\subset L_1$, if:

1) $D(L)\subset D(L_1)$,

2) $Lf=L_1f$ for all $f$ from $D(L)$.
\end{definition}

\begin{definition}
\label{def2}
A linear closed operator $L_0$ in a Hilbert space $H$ is called \textit{minimal} if $\overline{R(L_0)} \not=H$ and there exists a bounded inverse operator $L_0^{-1}$ on $R(L_0)$. 
\end{definition}

\begin{definition}
\label{def3}
A linear closed operator $\widehat{L}$ in a Hilbert space $H$ is called \textit{maximal} if $R(\widehat{L})=H$ and $\mbox{Ker}\, \widehat{L} \not=\{0\}$. 
\end{definition}

\begin{definition}
\label{def4}
A linear closed operator $L$ in a Hilbert space $H$ is called \textit{correct} if there exists a bounded inverse operator $L^{-1}$ defined on all of $H$. 
\end{definition}

\begin{definition}
\label{def5}
We say that a correct operator $L$ in a Hilbert space $H$ is a \textit{correct extension} of minimal operator $L_0$ (\textit{correct restriction} of maximal operator $\widehat{L}$) if $L_0\subset L$ ($L\subset \widehat{L}$).
\end{definition}

\begin{definition}
\label{def6}
We say that a correct operator $L$ in a Hilbert space $H$ is a \textit{boundary correct} extension of a minimal operator $L_0$ with respect to a maximal operator $\widehat{L}$ if $L$ is simultaneously a correct restriction of the maximal operator $\widehat{L}$ and a correct extension of the minimal operator $L_0$, that is, $L_0\subset L \subset \widehat{L}$.
\end{definition}

At the beginning of the 1950s, Vishik {\cite{Vishik}} extended the theory of self-adjoint extensions of von Neumann--Krein symmetric operators to nonsymmetric operators in Hilbert space.

At the beginning of the 1980s, M.~Otelbaev and his disciples proved abstract theorems that allows us to describe all correct extensions of some minimal operator using any single known correct extension in terms of an inverse operator. Here such extensions need not be restrictions of a maximal operator.
Similarly, all possible correct restrictions of some maximal operator that need not be extensions of a minimal operator were described (see {\cite{Kokebaev}}).
For convenience, we present the conclusions of these theorems.

Let $\widehat{L}$ be a maximal linear operator in a Hilbert space $H$, let $L$ be any known correct restriction of $\widehat{L}$, and let $K$ be an arbitrary linear bounded (in $H$) operator satisfying the following condition:
\begin{equation}\label{1.1}
R(K)\subset \mbox{Ker}\, \widehat{L}.
\end{equation}
Then the operator $L_K^{-1}$ defined by the formula 
\begin{equation}\label{1.2}
L_K^{-1}f=L^{-1}f+Kf,
\end{equation}
describes the inverse operators to all possible correct restrictions $L_K$ of $\widehat{L}$, i.e., $L_K\subset \widehat{L}$.

Let $L_0$ be a minimal operator in a Hilbert space $H$, let $L$ be any known correct extension of $L_0$, and let $K$ be a linear bounded operator in $H$ satisfying the conditions

a) $R(L_0)\subset \mbox{Ker}\,K$,

b) $\mbox{Ker}\,(L^{-1}+K)=\{0\}$,
\\
then the operator $L_K^{-1}$ defined by formula \eqref{1.2}
describes the inverse operators to all possible correct extensions $L_K$ of  $L_0$.

Let $L$ be any known boundary correct extension of $L_0$, i.e., $L_0\subset L\subset \widehat{L}$. The existence of at least one boundary correct extension $L$ was proved by Vishik in {\cite{Vishik}}. Let $K$ be a linear bounded (in $H$) operator satisfying the conditions

a) $R(L_0)\subset \mbox{Ker}\,K$,

b) $R(K)\subset \mbox{Ker}\,\widehat{L}$,
\\
then the operator $L_K^{-1}$ defined by formula \eqref{1.2}
describes the inverse operators to all possible boundary correct extensions $L_K$ of $L_0$. 

Self-adjoint and unitary operators are particular cases of normal operators. A bounded linear operator $N$ in a Hilbert space $H$ is called \textit{normal} if it commutes with its adjoint:
\[N^{*}N=NN^*. \]
The theory of bounded normal operators are sufficiently developed. 

Consider an unbounded linear operator $A$ in a Hilbert space $H$. 

\begin{definition} \label{def7} 
A densely defined closed linear operator $A$ in a Hilbert space $H$ is called \textit{formally normal} if
\[ D(A)\subset D(A^*), \quad \|Af\|=\|A^*f\| \quad \mbox{for all} \; f \in D(A). \]
\end{definition}

\begin{definition} \label{def8} 
A formally normal operator $A$ is called \textit{normal} if
\[D(A)=D(A^*).\]
\end{definition}

Normal extensions of formally normal operators have been studied by many authors (see {\cite{Biriuk}}, {\cite{Coddington2}}, {\cite{Coddington1}}, {\cite{Polyakov}}).
Questions the existence of a normal extension and the description of the domains of normal extensions of a formally normal operator were considered.

The spectral properties of the correct restrictions and extensions were systematically studied by the author (see \cite{Biyarov3}--\cite{Biyarov1}).
In these works a class of operators $K$ that provides Volterra, the completeness of root vectors, and the dissipativity of the correct restrictions and extensions were described. 
The present paper is devoted to the description of correct normal extensions in terms of the operator $K$.


\subsection{Coincidence criterion of $D(L)$ with $D(L^*)$}
\label{subsec2}

We consider a densely defined minimal linear operator $L_0$ in a Hilbert space $H$. Let  $M_0$ be a minimal operator  with $D(M_0)=D(L_0)$ that is connected with $L_{0}$ by the relation $(L_{0}u,v)=(u,M_{0} v)$ for all $u$,  $v$ from $D(L_0)$. Then the maximal operator $\widehat{L}=M_0^*$ is an extension of $L_{0}$, and the maximal operator $\widehat{M}=L_0^*$ is an extension of $M_{0}$. The following statement is true.

\begin{assertion} \label{assertion2.1} If there exists a correct extension $L_S$ of the minimal operator $L_0$ with the property $D(L_S)=D(L_S^*)$, then the operator $L_S$ is the boundary correct extension, i.e., $L_0\subset L_S \subset\widehat{L}$.
\end{assertion}

\begin{proof} From  $L_0 \subset L_S$ it follows that $L^*_S\subset L_0^* =\widehat{M}$. From  $D(L_S)=D(L_S^*)$ and the fact that $D(M_0)\subset D(L_S^*)$ we have
\[ M_0\subset L_S^*\subset \widehat{M}.\]
Then $L_0\subset L_S \subset\widehat{L}$. The assertion is proved.
\end{proof}

Let there be one fixed correct extension $L_S$ of $L_0$ such that $D(L_S)=D(L_S^*)$.
Then we can describe the inverses to all boundary correct extensions $L$ in the following form
\begin{equation}\label{2.1}
u= L^{-1}f=L_S^{-1}f+Kf \quad \mbox{for all} \; f\in H,
\end{equation}
where $K$ is an arbitrary bounded operator in a Hilbert space $H$ that
\[R(K)\subset \mbox{Ker}\,\widehat{L} \quad \mbox{and} \quad R(L_0)\subset \mbox{Ker}\, K.\]
Each such operator $K$ defines one boundary correct extension and there do not exist other boundary correct extensions.

Let us equip $D(\widehat{L})$ with the graph norm $||u||_G=(||u||^2+||\widehat{L}u||^2)^{1/2}$.
Since $\widehat{L}$ is a closed operator, we obtain a Hilbert space with the scalar product 
\[(u, v)_G=(u, v)+(\widehat{L}u, \widehat{L}v) \; \mbox{for all} \; u, \: v \; \mbox{from} \; D(\widehat{L}).\]
Let us denote this space by $G_{\widehat{L}}$. The domain $D(L_S)$ of the correct restriction $L_S$ is a subspace in $G_{\widehat{L}}$. Therefore, there exists a projection operator of $G_{\widehat{L}}$ on the subspace $D(L_S)$. As such a projection operator, we take $L_S^{-1}\widehat{L}$.
Then the projection $\Gamma_{L_S} = I-L_S^{-1}\widehat{L}$ of $G_{\widehat{L}}$ on the subspace $\mbox{Ker}\, \widehat{L}$. It is obvious that 
\[\mbox{Ker}\, \Gamma_{L_S} =D(L_S) \quad \mbox{and} \quad R(\Gamma_{L_S}) =\mbox{Ker}\, \widehat{L}.\]
All boundary correct extensions \eqref{2.1} transforms into
\[L^{-1}f=L_S^{-1}f+Kf=L_S^{-1}f+K\widehat{L}L_S^{-1}f=(I+K\widehat{L})L_S^{-1}f \; \mbox{for all} \; f \; \mbox{from} \; H,\]
where $I$ is the identity operator in $H$. In virtue of
$D(L) \subset D(\widehat{L})$, we have
\[ \widehat{L}u=f \; \mbox{for all} \; f \; \mbox{from} \;\: H, \;\; u\ \; \mbox{from} \;D(L)\]
where
\[ D(L)=\bigl\{ u\in D(\widehat{L}): \;(I-K\widehat{L})u\in D(L_S) \bigr\}.\]
It is easy to see that the operator $K$ defines the domain of $L$, as (see \cite{Biyarov2})
\[(I-K\widehat{L})D(L)=D(L_S), \quad (I+K\widehat{L})D(L_S)=D(L), \quad (I-K\widehat{L})=(I+K\widehat{L})^{-1}.\]
Therefore, all boundary correct extensions $L$ are differed from fixed boundary correct extension $L_S$ only the domain.
The bounded (in $G_{\widehat{L}}$) operator $I-K\widehat{L}$ maps $D(L)$ onto $D(L_S)$ in a one-to-one fashion.
Then the domain of $L$ can be defined as follows:
\[ D(L)=\bigl\{u\in D(\widehat{L}): \;\Gamma_{L_S}(I-K\widehat{L})u=0 \bigr\}.\]
There exists one more representation of the domain of $L$
\[ D(L)=\bigl\{u\in D(\widehat{L}): \;((I-K\widehat{L})u,L^*_S v)=(\widehat{L}u,v)\; \mbox{for all} \; v \; \mbox{from} \; D(L_S^*) \; \bigr\}.\]
Similarly we can define
\[ \Gamma_{L_S^*}=I-L_S^{*^{-1}}\widehat{M} \]
and
\[ D(L^*)=\bigl\{u\in D(\widehat{M}): \;\Gamma_{L_S^*}(I-K^*\widehat{M})u=0 \bigr\}.\]
Now we can formulate the following result:

\begin{theorem}\label{Theorem2.2} Let there be a correct extension $L_S$ of the minimal operator $L_{0}$ with $D(L_S)=D(L_S^*)$, then any other correct extension $L$ has the property $D(L)=D(L^*)$  if and only if $L_0\subset L\subset \widehat{L}$ and the operator $K$ from the formula \eqref{2.1} satisfies the conditions
\[R(K)\cup R(K^*)\subset D(\widehat{L})\cap D(\widehat{M}),\]
and
\begin{equation}\label{2.2}
\left \{
\begin{array}{lll}
& \Gamma_{L_S} (I-K\widehat{L})u=0, \\
& \Gamma_{L_S} K^*\widehat{M}u=K\widehat{L}u, \;\;\; \mbox{for all} \;\; u\in D(\widehat{L})\cap D(\widehat{M}), \\
\end{array}
\right.
\end{equation}
where $\Gamma_{L_S} =I-L_S^{-1}\widehat{L}$ is the projection defined above.
\end{theorem}

\begin{proof}
Let $D(L)=D(L^*)$. In view of Assertion \ref{assertion2.1}, the operators $L_S$ and $L$ turn out to be boundary correct extensions of $L_{0}$, i.e., $L_0\subset L_S\subset \widehat{L}$ and $L_0\subset L\subset \widehat{L}$. The inverse to the arbitrary boundary correct extension $L$ has the form \eqref{2.1}. Then
\[(L^*)^{-1}g=(L_S^*)^{-1}g+K^*g \quad \mbox{for all} \quad g\in H. \]
The condition $D(L)=D(L^*)$ is equivalent to 
\begin{equation}\label{2.3}
L_S^{-1}f+Kf=(L_S^*)^{-1}g+K^*g,
\end{equation}
where for each $f\in H$ there exists $g\in H$ and vice versa, for each $g\in H$ there exists $f\in H$ that the equality  \eqref{2.3} is fulfilled. 
It follows from  \eqref{2.3} that 
\[R(K^*)\subset D(\widehat{L}) \quad \mbox{and} \quad R(K)\subset D(\widehat{M}).\]
Then we get
\[R(K)\cup R(K^*)\subset D(\widehat{L})\cap D(\widehat{M}).\]
Acting on both sides of equality \eqref{2.3} by the operator $\widehat{L}$, we obtain
\[ f=L_S(L_S^*)^{-1}g+\widehat{L}K^*g, \quad \mbox{for all} \quad g\in H. \]
Substituting $f$ into \eqref{2.3}, we obtain the equality
\[ L_S^{-1}\widehat{L}K^*g+KL_S(L_S^*)^{-1}g+K\widehat{L}K^*g=K^*g. \]
It follows that
\[ (I-L_S^{-1}\widehat{L})K^*g=K\widehat{L}((L_S^*)^{-1}+K^*)g.\]
This means that
\[ (I-L_S^{-1}\widehat{L})K^*g=K\widehat{L}(L^*)^{-1}g. \]
If $L^{*^{-1}}g$ is replaced by $u$, then
\[(I-L_S^{-1}\widehat{L})K^*\widehat{M}u=K\widehat{L}u,\quad u\in D(L^*).\]
Since $D(L)= D(L^*)$ we obtain $\Gamma_{L_S} K^*\widehat{M}u=K\widehat{L}u $ for all $u$ from $D(L).$ This is equivalent to the condition \eqref{2.2}.

We now prove a converse of this theorem. Let $L_0\subset L\subset \widehat{L}$ and the operator $K$ from the formula \eqref{2.1} satisfies the conditions
$R(K)\cup R(K^*)\subset D(\widehat{L})\cap D(\widehat{M}),$
and \eqref{2.2}.
Hence, it is easy to see that
\[ D(L)\cup D(L^*)\subset D(\widehat{L})\cap D(\widehat{M}). \]
Since $Lu=f \; \mbox{for all} \; u \in D(L)$, we may replace $Lu$ by $f$ in the second equation of the condition \eqref{2.2}.
Then
\[ \Gamma_{L_S} K^*\widehat{M}L^{-1}f=Kf \quad \mbox{for all} \; f\in H. \]
Acting on both sides of this equality by the projection $\Gamma_{L_S^*}$, we obtain
\[ K^*\widehat{M}L^{-1}f=(I-(L_S^*)^{-1}\widehat{M})Kf \quad \mbox{for all} \; f\in H. \]
Note that
\[ K^*L_S^*L_S^{-1}f+K^*\widehat{M}Kf+(L_S^*)^{-1}\widehat{M}Kf=Kf. \]
Adding the bounded operator $L_S^{-1}f$ to both sides, we get
\[(L_S^*)^{-1}L_S^*L_S^{-1}f+K^*L_S^*L_S^{-1}f+K^*\widehat{M}Kf+(L_S^*)^{-1}\widehat{M}Kf=Kf+L_S^{-1}f.\]
It follows that
\[ (L_S^*)^{-1}(L_S^*L_S^{-1}+\widehat{M}K)f+K^*(L_S^*L_S^{-1}+\widehat{M}K)f=L^{-1}f \quad
\mbox{for all} \; f\in H. \]
If we denote by
\[ g=L_S^*L_S^{-1}f+\widehat{M}Kf \quad \mbox{for all} \; f \in H, \]
then we have
\[ (L^*)^{-1}g=L^{-1}f \quad \mbox{for all} \; f \in H. \]
It follows that $D(L)\subset D(L^*)$. 
Acting on both sides of the equations \eqref{2.2} by the projection $\Gamma_{L_S^*}$, we get
\[
\left \{
\begin{array}{lll}
& \Gamma_{L_S^*} (I-K\widehat{L})u=0, \\
& \Gamma_{L_S^*} K\widehat{L}u=K^*\widehat{M}u \;\;\; \mbox{for all} \;\; u\in D(\widehat{L})\cap D(\widehat{M}). \\
\end{array}
\right.
\]
By the second equation of the given system, we can rewrite this system of equations in the form
\[
\left \{
\begin{array}{lll}
& \Gamma_{L_S^*} (I-K^*\widehat{M})u=0, \\
& \Gamma_{L_S^*} K\widehat{L}u=K^*\widehat{M}u \;\;\; \mbox{for all} \;\; u\in D(\widehat{L})\cap D(\widehat{M}). \\
\end{array}
\right.
\]
The first equation of this system means that $u$ belongs to $D(L^*)$.
Then we denote $L^*u=g$. Therefore, $u=(L^*)^{-1}g$ for all $g$ from $H$.
Then the second equation of this system has the form
\[\Gamma_{L_S^*} K\widehat{L}(L^*)^{-1}g=K^*\widehat{M}(L^*)^{-1}g \;\;\; \mbox{for all} \;\; g\in H.\]
Acting on both sides of this equality by the projection $\Gamma_{L_S}$, we obtain
\[ K\widehat{L}(L^*)^{-1}g=(I-(L_S)^{-1}\widehat{L})K^*g \quad \mbox{for all} \; g\in H. \]
Note that
\[ KL_S(L_S^*)^{-1}g+K\widehat{L}K^*g+L_S^{-1}\widehat{L}K^*g=K^*g. \]
Adding the bounded operator $(L_S^*)^{-1}g$ to both sides, we get
\[L_S^{-1}L_S(L_S^*)^{-1}g+KL_S(L_S^*)^{-1}g+K\widehat{L}K^*g+L_S^{-1}\widehat{L}K^*g=K^*g+(L_S^*)^{-1}g.\]
It follows that
\[ L_S^{-1}(L_S(L_S^*)^{-1}+\widehat{L}K^*)g+K(L_S(L_S^*)^{-1}+\widehat{L}K^*)g=(L^*)^{-1}g \quad
\mbox{for all} \; g\in H. \]
If we denote by
\[ f=(L_S(L_S^*)^{-1}+\widehat{L}K^*)g \quad \mbox{for all} \; g \in H, \]
then we have
\[ L^{-1}f=(L^*)^{-1}g \quad \mbox{for all} \; g \in H. \]
It follows that $D(L^*)\subset D(L)$. 
The theorem is proved.
\end{proof}


\subsection{Normality criterion of correct extensions}
\label{section3}
Let $L_0$ be a formally normal minimal operator in a Hilbert space $H$. An operator $M_0$ is the restriction of $L_0^*=\widehat{M}$ to $D(L_0)$. Then $\widehat{L}=M_0^*$ defines the maximal operator that $L_0\subset\widehat{L}$.
Let there be at least one normal correct extension $L_N$ of the formally normal minimal operator $L_0$.
In view of Assertion \ref{assertion2.1}, we have that $L_0\subset L_N \subset \widehat{L}$, i.e., $L_N$ is the boundary correct extension.
Then the inverses to all boundary correct extensions $L$ of $L_0$ have the form
\begin{equation}\label{3.1}
u= L^{-1}f=L_N^{-1}f+Kf \quad \mbox{for all} \; f \in H,
\end{equation}
where $K$ is an arbitrary bounded operator in a Hilbert space $H$ that $R(K)\subset \mbox{Ker}\, \widehat{L}$ and $R(L_0)\subset \mbox{Ker}\, K$.
Then the direct operator $L$ acts as
\[ \widehat{L}u=f \quad \mbox{for all} \; f\in H, \]
on the domain
\[ D(L)=\bigl\{u\in D(\widehat{L}): \; \Gamma_{L_N} (I-K\widehat{L})u=0 \bigr\}, \]
where the projection $\Gamma_{L_N} =I-L_N^{-1}\widehat{L}$ is the bounded operator in the space $G_{\widehat{L}}$. It is known that
\[\mbox{Ker}\,\Gamma_{L_N}=D(L_N) \, \mbox{ and } \, R(\Gamma_{L_N})=\mbox{Ker}\, \widehat{L}.\]

\begin{theorem}
\label{Theorem3.1} Let there be one correct normal extension  $L_N$ of the formally normal minimal operator $L_0$ in a Hilbert space $H$. Then any other correct extension $L$ of $L_0$ is normal if and only if $L_0\subset L\subset \widehat{L}$ and operator $K$ from the formula \eqref{3.1} satisfies the conditions:
\[R(K)\cup R(K^*)\subset D(\widehat{L})\cap D(\widehat{M}),\]
\begin{equation}\label{3.2}
\left \{
\begin{array}{lll}
& \Gamma_{L_N} (I-K\widehat{L})u=0, \\
& \Gamma_{L_N} K^*\widehat{M}u=K\widehat{L}u \;\;\; \mbox{for all} \;\; u\in D(\widehat{L})\cap D(\widehat{M}), \\
\end{array}
\right.
\end{equation}
and
\begin{equation}\label{3.3}
\widehat{L}K^*=(\widehat{M}K)^*,
\end{equation}
where $\Gamma_{L_N}=I-L_N^{-1}\widehat{L}$ is projection on $\mbox{Ker}\, \widehat{L}$. 
\end{theorem}

\begin{proof} Let $L$ be a normal correct extension of the formally normal operator $L_0$. In view of Theorem \ref{Theorem2.2}, the conditions $L_0\subset L\subset \widehat{L}$, $R(K)\cup R(K^*)\subset D(\widehat{L})\cap D(\widehat{M})$ and
 \eqref{3.2} will be fulfilled. The normality of $L^{-1}$ follows from the normality of $L$:
\[ L^{-1}(L^*)^{-1}=(L^*)^{-1}L^{-1}. \]
By virtue of \eqref{3.1}, we obtain
\[
(L_N^{-1}+K)((L_N^*)^{-1}+K^*)f =((L_N^*)^{-1}+K^*)(L_N^{-1}+K)f \quad \mbox{for all} \;\; f \in H.
\]
It follows that
\begin{equation}\label{3.4}
L_N^{-1}K^*f+K(L_N^*)^{-1}+KK^*f =(L_N^*)^{-1}Kf+K^*L_N^{-1}f+K^*Kf.
\end{equation}
Acting on both sides of the equality \eqref{3.4} by the operator $\widehat{L}$, we get
\[ K^*f=L_N(L_N^*)^{-1}Kf+\widehat{L}K^*L_N^{-1}f+\widehat{L}K^*Kf. \]
Taking conjugates of both sides of the equality above, we have
\[ Kf=K^*(L_N(L_N^*)^{-1})^*f+(L_N^*)^{-1}(\widehat{L}K^*)^*f+ K^*(\widehat{L}K^*)^*f \quad \mbox{for all} \;\; f \in H. \]
Acting on both sides by the operator $\widehat{M}$, we obtain
\[ \widehat{M}Kf=(\widehat{L}K^*)^*f \quad \mbox{for all} \;\; f \in H. \]
This is equivalent to
\[ \widehat{L}K^*=(\widehat{M}K)^*. \]

Let us prove the converse. 
Suppose that the conditions of Theorem \ref{Theorem3.1} hold. 
From the conditions $L_0\subset L\subset \widehat{L}$, $R(K)\cup R(K^*)\subset D(\widehat{L})\cap D(\widehat{M})$ and \eqref{3.2}, in view of Theorem \ref{Theorem2.2}, we have that $D(L)=D(L^*)$. Then $\mbox{for all} \; f \in H$ there exists $g \in H$ such that $L^{-1}f=(L^*)^{-1}g$. It can be rewritten in the form
\begin{equation}\label{3.5}
L_N^{-1}f+Kf=(L_N^*)^{-1}g+K^*g.
\end{equation}
Acting on both sides by the operator $\widehat{M}$, we get
\[g=L_N^*L_N^{-1}f+\widehat{M}Kf.\]
Substituting $g$ into \eqref{3.5}, we have
\[ Kf=(L_N^*)^{-1}\widehat{M}Kf+K^*L_N^*L_N^{-1}f+K^*\widehat{M}Kf \quad \mbox{for all} \;\; f \in H. \]
Then
\begin{equation}\label{3.6}
K^*f=(\widehat{M}K)^*L_N^{-1}f+(L_N^*L_N^{-1})^*Kf+(\widehat{M}K)^*Kf \quad \mbox{for all} \;\; f \in H.
\end{equation}
Let us show that
\[ (L_N^*L_N^{-1})^*=L_N(L_N^*)^{-1}. \]
It is known that if $A$ is a closed operator, $B$ is bounded in $H$ and $AB$ is densely defined in $H$, then 
\[(AB)^*=\overline{B^*A^*},\] 
where the overbar denotes the closure operator. Note that
\[ L_N^*L_N^{-1}\supset L_N^{-1}L_N^*. \]
Then
\[ (L_N^*L_N^{-1})^*=\overline{(L_N^*)^{-1}L_N}\subset L_N(L_N^*)^{-1}. \]
Taking into account the fact that $L_N(L_N^*)^{-1}$ is the bounded operator that coincides with  $(L_N^*)^{-1}L_N$ on the dense set $D(L_N)$, then we obtain that
\[ L_N(L_N^*)^{-1}=\overline{(L_N^*)^{-1}L_N}= (L_N^*L_N^{-1})^*. \]
Then, taking into account \eqref{3.3}, the equality \eqref{3.6} can be rewritten in the form
\[ K^*f=\widehat{L}K^*L_N^{-1}f+L_N(L_N^*)^{-1}Kf+\widehat{L}K^*Kf \quad \mbox{for all} \;\; f \in H.\]
Adding $(L_N^*)^{-1}f$ to both sides of the last equality, we get
\[ K^*f+(L_N^*)^{-1}f=L_N L_N^{-1}(L_N^*)^{-1}f+ \widehat{L}K^*L_N^{-1}f+L_N(L_N^*)^{-1}Kf+\widehat{L}K^*Kf.\]
It follows that
\[ (L^*)^{-1}f=L(L^*)^{-1}L^{-1}f\;\;\; \mbox{for all} \;\; f \in H. \]
Thus
\[ L^{-1}(L^*)^{-1}f=(L^*)^{-1}L^{-1}f\;\;\; \mbox{for all} \;\; f \in H. \]
The proof is complete.
\end{proof}

The domain of $L_S$ described as the kernel of the projection $\Gamma_{L_S}=I-L_S^{-1}\widehat{L}$. Here the operator $L_S^{-1}$ takes part in the explicit form. Sometimes there exists another operator $T_{L_S}$ defined on $D(\widehat{L})$ and has the property $\mbox{Ker}\, \Gamma_{L_S}=\mbox{Ker}\, T_{L_S}$. Between these operators have the following relationship
\[ T_{L_S}\Gamma_{L_S} v=T_{L_S}(I-L_S^{-1}\widehat{L})v=T_{L_S}v-T_{L_S}L_S^{-1}\widehat{L}v=T_{L_S}v \quad \mbox{for all} \;\; v\in D(\widehat{L}). \]
If we know $T_{L_S}v$, then $\Gamma_{L_S} v$ is uniquely determined as the solution of the homogeneous equation $\widehat{L} (\Gamma_{L_S} v)=0$ with an inhomogeneous condition
\[ T_{L_S}(\Gamma_{L_S} v)=T_{L_S}v. \]
Its unique solvability follows from the correctness of the operator $L_S$.
Therefore, it is not necessary to know the explicit form of the operator $L_S^{-1}$. In the study of differential operators (see \cite{Vishik}) that the operator $T_{L_S}$ is realized in the form of the boundary operator.
In such cases we say that the domain is described in terms of the boundary operator. For example, in the case of the Dirichlet problem for a differential equation of elliptic type in $L_2(\Omega)$ that $T_{L_S}$ corresponds to the trace operator on the boundary of $\Omega$, i.e., $T_{L_S}u=u\mid _{\partial \Omega}$.
Therefore it is sufficient to know the form of the boundary operator $T_{L_S}$.  
Thus we obtain the following

\begin{corollary}\label{corollary3.2}
Let there be a correct extension $L_S$ of the minimal operator $L_{0}$ with $D(L_S)=D(L_S^*)$, then any other correct extension $L$ has the property $D(L)=D(L^*)$  if and only if $L_0\subset L\subset \widehat{L}$, $R(K)\cup R(K^*)\subset D(\widehat{L})\cap D(\widehat{M})$ and
\begin{equation}\label{3.7}
 T_{L_S}(K^*\widehat{M}-K\widehat{L})u=0
 \quad \mbox{for all} \;\; u\in D(L),
\end{equation}
where $T_{L_S}$ is a boundary operator corresponding to the fixed correct extension $L_S$ and
\[ D(L)=\bigl \{u\in D(\widehat{L}): \; T_{L_S}(I-K\widehat{L})u=0 \bigr\}. \]
\end{corollary}

\begin{remark}\label{remark3.3}
By virtue of the one-to-one mapping of $D(L_S)$ onto $D(L):$
\[ v=(I-K\widehat{L})u \;\; \mbox{for all} \;\; u\in D(L), \quad u=(I+K\widehat{L})v \;\; \mbox{for all} \;\; v\in D(L_S), \]
in practice, sometimes it is more convenient to use the following condition that is equivalent to \eqref{3.7}: 
\begin{equation}\label{3.8}
 T_{L_S}(K^*\widehat{M}-K\widehat{L}+K^*\widehat{M}K\widehat{L})v=0 \quad \mbox{for all} \;\; v\in D(L_S).
\end{equation}
It has the practical convenience because $D(L_S)$ is a fixed domain.
\end{remark}
Similarly, we can rephrase Theorem \ref{Theorem3.1} in the following form

\begin{corollary}\label{corollary3.4}
Let there be one correct normal extension  $L_N$ of the formally normal minimal operator $L_0$ in a Hilbert space $H$. Then any other correct extension $L$ of $L_0$ is normal if and only if $L_0\subset L\subset \widehat{L}$, $R(K)\cup R(K^*)\subset D(\widehat{L})\cap D(\widehat{M})$,
\begin{equation}\label{3.9}
 T_{L_N}(K^*\widehat{M}-K\widehat{L})u=0 \quad \mbox{for all} \;\; u\in D(L),
\end{equation}
and
\begin{equation}\label{3.10}
 \widehat{L}K^*= (\widehat{M}K)^*,
\end{equation}
where $T_{L_N}$ is a boundary operator corresponding to the fixed correct extension $L_N$ and
\[ D(L)=\bigl \{u\in D(\widehat{L}): T_{L_N}(I-K\widehat{L})u=0 \bigr\}, \]
and $K$ is the operator determining the boundary correct extension $L$ from the formula \eqref{3.1}.
\end{corollary}

\subsection{The Examples}
\label{section4}
\begin{example}\label{example4.1}
We consider the following operator in a Hilbert space $L_2(0,1)$
\begin{equation}\label{4.1}
\widehat{L}y\equiv y''+y'=f,
\end{equation}
to which corresponds the minimal operator $L_0$ with domain
\[ D(L_0)=\bigl\{y\in W_2^2(0,1): \; y(0)=y(1)=y'(0)=y'(1)=0 \bigl\}.\]
\end{example} 
We define the operator $M_0$ as the restriction of $\widehat{M}$ on the set $D(L_0)$. Then the action of the operator $M_0$ has the form
\[ \widehat{M}y \equiv y''-y'=f.\]
We will denote the maximal operators $M^*_0$ and $L^*_0$ by $\widehat{L}$ and $\widehat{M}$, respectively. Then we have 
\[L_0 \subset \widehat{L},\,\,M_0 \subset \widehat{M} \quad \mbox{and} \quad D(\widehat{L})= D(\widehat{M}) = W_2^2(0,1).\] 
Let the operator $L_N$ acts as $\widehat{L}$ with domain
\[ D(L_N)=\bigl\{y\in D(\widehat{L}): \; y(0)+y(1)=0,\,\,y'(0)+y'(1)=0 \bigr\}.\]
We take the operator $L_N$ as the fixed correct extensions of $L_0$. 
Note that $D(L_N)=D(L_N^*)$ and $L_0 \subset L_N \subset \widehat{L}, \; M_0 \subset L_N^* \subset \widehat{M}$.
The inverse operator to $L_N$ has the form
\[
y= L_N^{-1}f=\int\limits^{x}\limits_{0}(1-e^{t-x})f(t)dt-\frac{1}{2}\int\limits^{1}\limits_{0} f(t)dt
+\frac{e^{1-x}}{1+e}\int\limits^{1}\limits_{0}e^{t-1}f(t)dt.
\]
Then $\Gamma_{L_N}$ is defined as
\[
\Gamma_{L_N} y=\frac{y(0)+y(1)}{2}+\biggl(\frac{1}{2}-\frac{e^{1-x}}{1+e}\biggr)[y'(0)+y'(1)].
\]
And $\Gamma_{L_N^*}=I-(L_N^*)^{-1}\widehat{M}$ has the following form
\[
\Gamma_{L_N^*}y=\frac{y(0)+y(1)}{2}+\biggl(\frac{e^x}{1+e}-\frac{1}{2}\biggr)[y'(0)+y'(1)].
\]
The correct extension $L$ of $L_0$ with the property $D(L)=D(L^*)$ is a boundary correct  extension. Their inverses are described in the following form
\[ y= L^{-1}f=L_N^{-1}f+Kf \quad \mbox{for all} \;\; f \in L_2(0,1),\]
where $K$ is a bounded linear operator in $L_2(0,1)$ with the properties
\[ R(K)\subset \mbox{Ker}\, \widehat{L}, \quad R(L_0)\subset \mbox{Ker}\, K. \]
In our case, such operators are exhausted by the following operators
\[ Kf=\int\limits^{1}\limits_{0}f(t)(\overline{a_{11}} +\overline{a_{12}}e^t)dt+e^{-x}\int\limits^{1}\limits_{0}f(t)(\overline{a_{21}}+\overline{a_{22}}e^t)dt,
\]
where $a_{ij},\,\,i,j=1,2$ are arbitrary complex numbers.
Then
\[ K^*f=(a_{11}+a_{12}e^x)\int\limits^{1}\limits_{0}f(t)dt +(a_{21}+a_{22}e^x)\int\limits^{1}\limits_{0} e^{-t}f(t)dt.
\]
It is known that the direct operator $L$ acts as $\widehat{L}$ from \eqref{4.1} and the domain has the form
\[ D(L)=\bigl\{y\in D(\widehat{L}): \; \Gamma_{L_N}(I-K\widehat{L})y=0 \bigr\}. \]
In view of Corollary \ref{corollary3.2}, the domain of $L$ can be defined in another way
\[ \begin{split} 
D(L)=\biggl\{y\in D(\widehat{L}): \; &y(0)+y(1)=(K\widehat{L}y)(0)+(K\widehat{L}y)(1), \\
&y'(0)+y'(1)=\Bigl(\frac{d}{dx}K\widehat{L}y \Bigr)(0)+\Bigl(\frac{d}{dx}K\widehat{L}y \Bigr)(1) \biggr\}.
\end{split} \]
First, we will find the correct extensions $L$ such that $D(L)=D(L^*)$. Taking into account Remark \ref{remark3.3}, let the operator $K$ satisfies the condition \eqref{3.8}. Then we obtain the system of equations:
\[
\left \{
\begin{array}{rl}
&4(a_{11}+\overline{a_{11}})+2(e+1) \Bigl[\dfrac{\overline{a_{21}}}{e}+a_{12} \Bigr]\cdot A=0, \\
\\
&-4(a_{11}-\overline{a_{11}})-2(e+1)(a_{12}-\overline{a_{12}})-2\dfrac{e+1}{e}(a_{21}-\overline{a_{21}})\\
& \qquad \qquad -\dfrac{(e+1)^2}{e}(a_{22}-\overline{a_{22}})+\Bigl[4a_{12}+2\dfrac{e+1}{e}a_{22} \Bigr] \cdot A=0, \\
\\
&-\dfrac{1}{e}\overline{a_{21}}+a_{12}+\dfrac{2}{e}a_{12} \Bigl[\overline{a_{21}}(e-1)+\overline{a_{22}}\dfrac{e^2-1}{2} \Bigr]=0, \\
\\
&-\dfrac{1}{e}[2\overline{a_{21}}+\overline{a_{22}}(1+e)]-2a_{12}-\dfrac{e+1}{e}a_{22}\\
& \qquad \qquad -4\dfrac{a_{12}}{e} \Bigl[\overline{a_{21}}+
\overline{a_{22}}\dfrac{e^2-1}{2} \Bigr]-2\dfrac{e+1}{e^2}a_{22} \Bigl[\overline{a_{21}}(e-1)+\overline{a_{22}}\dfrac{e^2-1}{2}\Bigr]=0,
\end{array} \right. \]
where
\[
A=2(e-1)\overline{a_{11}}+(e^2-1)\overline{a_{12}}+\frac{e+1}{e} \Bigl[\overline{a_{21}}(e-1)+\overline{a_{22}}\frac{e^2-1}{2} \Bigr].
\]
Solutions of the system of equations with respect to $a_{ij},\,\,\,i,j=1,2,$ define the operators $K$ that guarantees the equality $D(L)=D(L^*)$. They will correspond to the following cases:
\begin{equation*}
\begin{split}
&I)\;\;\;D(L)=\Bigl\{y\in D(\widehat{L}): \; y(0)=0, \quad y(1)=0 \Bigr\},\\
&II)\;\;D(L)=\Bigl\{y\in D(\widehat{L}): \; y(0)=\frac{a-i}{a+i}y(1), \quad y'(0)=\frac{a-i}{a+i}y'(1), \quad a\in \mathbb{R}, \\
& \qquad \qquad \qquad \mbox{where} \; \mathbb{R} \: \mbox{is the space of real numbers} \Bigr\},\\
&III)\;D(L)=\Bigl\{y\in D(\widehat{L}): \; ay(0)+\bar{b}y(1)=0, \quad y(1)=by'(0)+ay'(1),\\
& \qquad \quad a\in \mathbb{R}, \quad a\neq 0, \quad b \in \mathbb{C}, \quad |b|^2=a^2, \;\; \mbox{where} \; \mathbb{C} \: \mbox{is the space of complex numbers} \Bigr\}.
\end{split}
\end{equation*}

We use the criterion given in Theorem \ref{Theorem3.1} to find all correct normal extensions $L$ of the minimal operator $L_0$. It is easy to verify the formal normality of $L_0$ and the normality of $L_N$.
The equality $D(L)=D(L^*)$ is necessary for the normality of $L$. 
They correspond to three cases of $I)-III)$ described above. Now, if the operator $K$ satisfies \eqref{3.3}, then the operator $L$ is a normal. The condition \eqref{3.3} is equivalent to the following
\[a_{21}=0, \quad a_{12}=0.\]
Therefore, the operator $K$ takes the form
\[ Kf=\overline{a_{11}}\int\limits^{1}\limits_{0}f(t)dt+
\overline{a_{22}}e^{-x}\int\limits^{1}\limits_{0}e^{t}f(t)dt.\]
Then operators $L$ which act as $\widehat{L}$ from \eqref{4.1} turn out to be the normal correct extensions and with the domain
\[ D(L)=\Bigl\{y\in D(\widehat{L}): \; y(0)=\frac{a-i}{a+i}y(1), \quad y'(0)=\frac{a-i}{a+i}y'(1), \quad a\in \mathbb{R} \Bigr\}. \]
From three cases of $I)-III)$ are suitable only the case $II)$.

\begin{example}\label{example4.2} Let in the Hilbert space $L_2(\Omega)$, where $\Omega=\{(x,y): \; 0<x<1,\; 0<y<1\}$, we consider the minimal operator $L_0$ generated by the Cauchy-Riemann differential operator 
\begin{equation}\label{4.2}
\widehat{L}u\equiv \frac{\partial u}{\partial x}+i\frac{\partial u}{\partial y} =f(x,y).
\end{equation}
Then
\[ D(L_0)=\bigl\{u\in W_2^1(\Omega): \; T_{L_0}u=0 \bigr\}, \]
where $T_{L_0}$ is a boundary operator defined as the trace of function $u\in W_2^1(\Omega)$ on the boundary of $\partial\Omega$. 
\end{example}

The action of $\widehat{M}$ will have the form
\[ \widehat{M}u\equiv-\frac{\partial u}{\partial x}+i\frac{\partial u}{\partial y}=f(x,y).\]
Domains of the operators $\widehat{L}$ and $\widehat{M}$ have the form
\[ D(\widehat{L})=\bigl\{u\in L_2(\Omega): \; \widehat{L}u\in L_2(\Omega) \bigr\},\]
\[ D(\widehat{M})=\bigl\{u\in L_2(\Omega): \; \widehat{M}u\in L_2(\Omega) \bigr\},\]
respectively.
If we define the boundary operator $T_{L_N}$ the following way
\[
T_{L_N}u=\left(%
\begin{array}{c}
  u(0,y)+u(1,y) \\
  u(x,0)+u(x,1) \\
\end{array}%
\right) \;\;\;\; \mbox{for all} \;\; u\in D(\widehat{L}),
\]
then the operator $L_N$ acting as $\widehat{L}$ with the domain
\[ D(L_N)=\bigl\{u\in D(\widehat{L}): \; T_{L_N}u=0 \bigr\},\]
is the correct extension of $L_0$. 
It is easy to verify that $L_0$ is formally normal and $L_N$ is normal, and in addition $L_0\subset L\subset \widehat{L}$.

We are interested in the normal boundary correct extensions.
Let us clarify some properties of the operator $K$:
\begin{equation*}
\begin{split}
&1)\: R(K) \subset W_2^1(\Omega);\\
&2)\: (Kf)(x+iy);\\
&3)\: (K^*f)(x-iy).
\end{split}
\end{equation*}
The first property follows from the fact that

\begin{assertion}\label{assertion4.3}
The domain of any normal correct extension $L$ of the minimal operator $L_0$ generated by the differential operator
\eqref{4.2} has the property: 
\[ D(L) \subset W_2^1(\Omega).\]
\end{assertion}

\begin{proof} It follows from Theorem 2 of Plesner and Rohlin (see {\cite{Plesner}}). Now we formulate this theorem: "For each pair of adjoint normal operators $A$ and $A^*$ there exists one and only one pair of self-adjoint operators $A_1$ and $A_2$, satisfying the condition
\[ A=A_1+iA_2,\;\;\;\;A^*=A_1-iA_2, \]
where the operators $A_1$ and $A_2$ commute".
\end{proof}

The second property follows from the condition $R(K) \subset \mbox{Ker}\, \widehat{L}$. The third property follows from the condition $R(L_0) \subset \mbox{Ker}\, K$.
Further from the conditions \eqref{3.9} and \eqref{3.10} obtain  the operators $K$ for which the correct boundary extension $L$ will be normal.

It follows from Assertion \ref{assertion4.3} that $L_N^{-1},\;K$, and $L^{-1}$ are compact operators in $L_2(\Omega)$. This means that the normal correct extension $L$ of $L_0$ is the operator of the discrete spectrum. Hence we have that $L$ has a complete orthonormal system of eigenfunctions.

For clarity, the check of normality by Theorem \ref{Theorem3.1}, we consider the special case. Let $K$ will be an integral operator of the form
\[ Kf=\int\limits^{1}\limits_{0}\int\limits^{1}\limits_{0} \mathcal{K}(x,y; \,\xi,\eta)f(\xi,\eta)d\xi d\eta. \]
It follows from properties 1) and 2) that
\[ Kf=\int\limits^{1}\limits_{0}\int\limits^{1}\limits_{0} \mathcal{K}(x+iy, \,\xi+i\eta)f(\xi,\eta)d\xi d\eta. \]
From the condition \eqref{3.3} of Theorem \ref{Theorem3.1}, we get that
\[ Kf=\int\limits^{1}\limits_{0}\int\limits^{1}\limits_{0} \mathcal{K}(x-\xi+i(y-\eta))f(\xi,\eta)d\xi d\eta. \]
Using the condition \eqref{3.2} of Theorem \ref{Theorem3.1} for the operator $K$, we obtain all normal correct extensions.
We will not give this condition on the kernel $\mathcal{K}(x-\xi+i(y-\eta))$, because of the cumbersome to write.

To demonstrate the mechanism of checking the condition \eqref{3.2}, we consider the special case when
\[ \mathcal{K}(x-\xi+i(y-\eta))=ae^{i\pi(x-\xi+i(y-\eta))},\]
where $a\in \mathbb{C}$ is a complex number of the form $a=a_1+ia_2.$ Then the condition \eqref{3.2} is equivalent to
\[ 2a_2+(a_1^2+a_2^2)(e^\pi-e^{-\pi})=0. \]
There are two kinds of solutions of this equation:
\begin{equation*}
\begin{split}
&I.\;\;\;a_1=0,\;\;\; a_2=\frac{2}{e^{-\pi}-e^\pi};\\
&II.\;\; a_2=\frac{-1\pm \sqrt{1-[a_1(e^\pi-e^{-\pi})]^2}}{e^\pi-e^{-\pi}}, \quad \mbox{where} \quad |a_1|\leq \frac{1}{e^\pi-e^{-\pi}}.\\
\end{split}
\end{equation*}
Then in the case of $II$, the correct extension corresponding to the following boundary problem
\begin{equation*}
\begin{split}
&\widehat{L}u\equiv \frac{\partial u}{\partial x}+i\frac{\partial u}{\partial y}=f(x,y) \quad \mbox{for all} \;\; f\in L_2(\Omega),\\
&D(L)=\biggl\{u\in W_2^1(\Omega): \; u(0,y)+u(1,y)=0, \quad 0\leq y\leq1,\\
& \qquad \qquad u(x,0)+u(x,1)=ia(e^\pi+1)\int\limits^{1}\limits_{0}e^{i\pi(x-\xi)}u(\xi,1)d\xi\\
& \qquad \qquad -ia(e^{-\pi}+1)\int\limits^{1}\limits_{0}e^{i\pi(x-\xi)}u(\xi,0)d\xi, \quad
0\leq x\leq1 \biggr\}
\end{split}
\end{equation*}
is normal, where $a=a_1+ia_2$,  or in the case of $I$, the correct extension corresponding to the boundary problem
\[
\begin{split}
& D(L)=\biggl\{u\in W_2^1(\Omega): \; u(0,y)+u(1,y)=0,\\ 
& \qquad \qquad u(x,0)+u(x,1)=2\int\limits^{1}\limits_{0}e^{i\pi(x-\xi)}u(\xi,1)d\xi\biggr\},
\end{split}
\]
is normal.

All normal correct extensions $L$ have a compact inverse operator because of $D(L) \subset W_2^1(\Omega)$. Therefore, their eigenfunctions create an orthonormal basis in $L_2(\Omega)$. In the particular case when
\[ \mathcal{K}(x,y; \,\xi,\eta)=\frac{2i}{e^{-\pi}-e^{\pi}}\cdot e^{i\pi(x-\xi+i(y-\eta))}, \]
we obtain the orthonormal basis in the following form:
\[
u_{k,n}(x,y)=\left \{
\begin{array}{lll}
& e^{2n\pi iy+i\pi x}, \quad &n=0,\pm1,\pm2,\ldots \\
& e^{(2k+1)\pi ix+(2n+1)\pi iy}, \quad &k=\pm1,\pm2,\ldots, \; n=0,\pm1,\pm2,\ldots \\
\end{array}
\right.
\] 
and the corresponding eigenvalues
\[
\lambda_{k,n}=\left \{
\begin{array}{lll}
& i\pi -2n\pi, \quad \quad &n=0,\pm1,\pm2,\ldots \\
& (2k+1)\pi i-(2n+1)\pi, \quad &k=\pm1,\pm2,\ldots, \; n=0,\pm1,\pm2,\ldots. \\
\end{array}
\right.
\] 

Thus, this method allows us to check for normality of an unbounded operator. Preliminary it is necessary to clarify the question of the existence of at least one normal extension.
For the existence of a normal extension we need that the minimal operator must be formally normal.

\begin{remark}\label{remark4.4}
If in Example \ref{example4.2} the square area $\Omega $ is replaced by the unit circle, then the minimal operator $L_0$ will not be formally normal. Thus in this case,  there are no normal extensions of $L_0$ in $L_2(\Omega)$. 
\end{remark}

\begin{remark}\label{remark4.5}
When the minimal operator $L_0$ is symmetric and the fixed operator $L_N$ is self-adjoint then the conditions of Theorem \ref{Theorem3.1} are equivalent to $K=K^*$ and we have all the self-adjoint correct extensions.
\end{remark}

\vskip 0.5 cm

\newpage

\vskip 0.5 cm

\begin{flushleft}
   Bazarkan Nuroldinovich Biyarov \\
   Department of Fundamental Mathematics\\
   L.N. Gumilyov Eurasian National University \\
   2 Mirzoyan St,\\
   010008 Astana, Kazakhstan\\
   E-mail: bbiyarov@gmail.com
\end{flushleft} 

\end{document}